\documentclass[psamsfonts]{amsart}

\usepackage{amssymb}

\usepackage{graphicx}
\usepackage{epstopdf}

\usepackage{amsmath}
\usepackage{amsthm}
\usepackage{amssymb}
\usepackage{mathrsfs}
\usepackage{dsfont}
\usepackage{enumerate}
\usepackage{color}
\usepackage[left=3cm,right=3cm, bottom=3cm]{geometry}
\usepackage[bookmarks,colorlinks]{hyperref}
\numberwithin{equation}{section}

\theoremstyle{theorem}
\newtheorem{theorem}{Theorem}[section]

\newtheorem{lemma}[theorem]{Lemma}
\theoremstyle{definition}

\theoremstyle{remark}
\theoremstyle{proof}

\newcommand{\C}{\mathbb{C}}

\newcommand{\R}{\mathbb{R}}

\newcommand{\cF}{\mathcal{F}}

\newcommand{\cL}{\mathcal{L}}
\newcommand{\cM}{\mathcal{M}}

\newcommand{\cS}{\mathcal{S}}

\newcommand{\sF}{\mathscr{F}}

\newcommand{\wh}{\widehat}
\newcommand{\wt}{\widetilde}
\renewcommand{\Re}{\operatorname{Re}}

\newcommand{\E}{\mathbb{E}}
\renewcommand{\P}{\mathbb{P}}

\renewcommand{\L}{\mathbb{L}}
\newcommand{\indi}[1]{\mathds{1}_{#1}}

\newcommand{\varep}{\varepsilon}
\newcommand{\ol}{\overline}

\newcommand{\norm}[1]{\| #1 \|}

\newcommand{\Paren}[1]{\left( #1 \right)}

\newcommand{\Abs}[1]{\left| #1 \right|}
\newcommand{\brk}[1]{\langle #1 \rangle}
\newcommand{\Ebrx}[2]{\E^{#1}\left[ #2\right]}

\newcommand{\parder}[2]{\frac{\partial #1 }{\partial #2 }}

\begin{document}
\title{Hardy-Stein identity for non-symmetric \\L\'vey processes and Fourier Multipliers}


\author{Rodrigo Ba\~nuelos}
\address{Department of Mathematics, Purdue University\\
150 N. University Street, West Lafayette, IN 47907-2067, USA}
\email{banuelos@purdue.edu}

\author{Daesung Kim}
\address{Department of Mathematics, Purdue University\\
150 N. University Street, West Lafayette, IN 47907-2067, USA}
\email{kim1636@purdue.edu}
\thanks{\it Both authors were supported in part  by NSF grant
\#1403417-DMS;  R. Ba\~nuelos PI}

\begin{abstract}
Using It\^o's formula for processes with jumps, we extend the Hardy-Stein identity proved in \cite{BBL} to non-symmetric L\'evy processes and derive its martingale version.  By  a symmetrization argument for Littlewood-Paley functions, the  $L^p$ boundedness of Fourier multipliers arising from  non-symmetric L\'evy processes is proved.    
\end{abstract}
\maketitle

\section{Introduction}

Littlewood-Paley square (quadratic)  functions have been of interest for  many years with many applications in harmonic analysis and probability.  On the analysis side, these include the classical square functions obtained from the Poisson semigroup as in \cite{Stein} and more general heat semigroups as in \cite{Stein0}.  On the probability side, these correspond to the celebrated Burkholder-Gundy inequalities which are the bread and butter of modern stochastic analysis. For a short review of some of this literature, we refer the reader to \cite{BBL}.  In \cite{BBL}, the authors extend some of the classical  Littlewood-Paley $L^p$ inequalities  for $1<p<\infty$ to non-local operators and then  apply them to prove $L^p$ bounds for certain Fourier multipliers that arise from transformation of symmetric L\'evy processes.  The symmetry of the L\'evy measure is crucial in the proof of the Littlewood-Paley inequalities for the full range of $1<p<\infty$.  

The key to the proof in \cite{BBL}, outside of fairly standard and well known uses of the Burkholder-Gundy inequalities and other arguments as in \cite{Stein}, is a Hardy-Stein identity  (\eqref{maineq} below) which is proved from properties of the semigroup.   In the classical case of the Laplacian, such Hardy-Stein  identity   follows from, essentially, Green's theorem and the chain rule as in Lemmas 1 and 2 in \cite[pp.86-87]{Stein}. In the case of Brownian motion, a probabilisitic Burkholder-Gundy type version of this Hardy-Stein  identity  can  be proved (see \cite{Ban}, \cite[p.152]{RavYor}) as a simple application of It\^o's formula.    While it is unknown (see \cite[p.265]{Da}) whether It\^o's  formula arguments give the Burkholder-Gundy inequalities for  general martingales with jumps,  the proof in \cite{Ban} can be adapted to our case of L\'evy processes studied here to prove the Hardy-Stein identity, Theorem \ref{taylorexpansion}.  In fact, our proof extends the ra the Hardy-Stein identity to non--symmetric  L\'evy measures and it gives a version, Theorem \ref{martHardyStein}, for martingales.  Finally, the proof contains additional information, further  illuminating the origins of the function $F(a,b;p)$ (see \eqref{eq:F}) used in \cite{BBL}.  

It is important to emphasize here that although the Hardy-Stein identity holds for non-symmetric L\'evy measures, the full comparability of the $L^p$-norms between the function itself and its Littlewood-Paley square function proved in \cite{BBL} requires symmetry and hence the main application given there to the boundedness of the Fourier multipliers requires it too.   In this paper, we show that the Littlewood-Paley proof can still be used to obtain the $L^p$ boundedness of the Fourier multipliers for non-symmetric L\'evy measures. Although  this argument uses the basic symmetrization technique as in \cite{BBB}, the application here goes via a symmetrization of the Littlewood-Paley function; see \eqref{defsquareftn} and \eqref{symG*}.   For more details on these types of Fourier multipliers which are, probabilistically speaking, natural extensions of those obtained from the conditional expectations of martingale transform as in \cite{Ban}, we refer the reader to \cite{AB}.

The paper is organized as follows.  The Hardy-Stein identity  via It\^o's formula is proved in \S3 where we also note that it holds for more general martingales.   In \S4, we prove the $L^p$ inequalities  for the Fourier multipliers arising from non-symmetric L\'evy processes.  

\section{Preliminaries}

The indicator function of a set $A$ is denoted by $\indi{A}$. For $a,b\in\R$, we denote by $a\wedge b = \min\{a,b\}$. The real part of a complex number $\xi$ is denoted by $\Re(\xi)=x$ where $\xi=x+iy$. For a set $B\subseteq \R^{d}$, we define $-B:=\{-x:x\in B\}$. An open ball in $\R^{d}$ of radius $r$, centered at $x_{0}\in\R^{d}$ is denoted by $B_{r}(x_{0})$. If $x_{0}=0$, we simply denote by $B_{r}$. For $f,g\in L^{2}(\R^{d})$, we define the inner product of $f$ and $g$ in $L^{2}(\R^{d})$ by $\brk{f,g}=\int_{\R^{d}}f(x)g(x)dx$.
Let $\cS(\R^{d})$ be the Schwartz space on $\R^{d}$ and $f\in \cS(\R^{d})$. We define the Fourier transform and the inverse Fourier transform of $f$ by 
\begin{eqnarray*}
	&&\sF(f)(\xi)=\wh{f}(\xi):=\int_{\R^{d}}f(x)e^{-ix\cdot \xi}dx,\\
	&&\sF^{-1}(f)(x)=f^{\vee}(x)=(2\pi)^{-d}\int_{\R^{d}}f(\xi)e^{i\xi\cdot x}d\xi.
\end{eqnarray*}
With our definition, Parseval's formula takes the form 
\begin{eqnarray}\label{parseval}
	\int_{\R^{d}}f(x)g(x)dx=\frac{1}{(2\pi)^{d}}\int_{\R^{d}}\wh{f}(\xi)\ol{\wh{g}(\xi)}d\xi, 
\end{eqnarray}
for $f,g\in L^{2}(\R^{d})$.

A $d$-dimensional stochastic process $(X_{t})_{t\geq0}$  defined on a filtered probability space $(\Omega,\cF,\P)$ is called  a \emph{L\'{e}vy process} if 
\begin{enumerate}[(i)]
	\item for $0\leq t_{0} < t_{1}< \cdots< t_{n}<\infty$, $\{X_{t_{k}}-X_{t_{k-1}}\}_{k\geq 1}$ are independent,
	\item for $0<s<t<\infty$ and a Borel set $A\subseteq \R^{n}$, $\P(X_{t}-X_{s}\in A)=\P(X_{t-s}\in A)$ and
	\item for all $\varep>0$ and $s\geq0$,
	\begin{eqnarray*}
		\lim_{t\to s}\P(|X_{t}-X_{s}|>\varep)=0.
	\end{eqnarray*}
\end{enumerate}
The characteristic exponent $\psi(\xi)$ of a L\'evy process $(X_{t})_{t\geq0}$ is defined by $\E[e^{i\xi\cdot X_{t}}]=e^{-t\psi(\xi)}$. The L\'{e}vy-Khintchine theorem tells us that $(X_{t})_{t\geq0}$ is a L\'evy process with characteristic exponent $\psi(\xi)$ if and only if there exists a triplet $(b,A,\nu)$ such that
\begin{eqnarray*}
	\psi(\xi)=ib\cdot \xi+\frac{1}{2}\xi\cdot A\xi
		+\int_{\R^{d}}(1-e^{i\xi\cdot y}+i\xi\cdot y \indi{B_{1}}(y))\nu(dy), 
\end{eqnarray*}
where $b\in\R^{d}$, $A$ is a positive semi-definite $d\times d$ matrix, and $\nu$ is a measure on Borel sets in $\R^{d}$ satisfying
\begin{eqnarray*}
	\int_{\R^{d}\setminus \{0\}}(1\wedge|y|^{2})\nu(dy)<\infty.
\end{eqnarray*}
We call $\nu$ the \emph{L\'{e}vy measure}. This gives a large class of stochastic processes that have been extensively studied. For instance, Brownian motion is the case where $b=0$, $\nu=0$, and $A$ is the identity matrix. We say that $(X_{t})_{t\geq0}$ is a \emph{pure jump L\'{e}vy process} if $b=0$ and $A=0$. We also say that $(X_{t})_{t\geq0}$ is \emph{symmetric} if $\nu$ is symmetric. We refer the reader to \cite{Da} for further information on these processes.  

Let $(X_{t})_{t\geq0}$ be a L\'evy process with the L\'evy measure $\nu$. The jump of $X_{t}$ at time $s$ is denoted by $\Delta X_{s}=X_{s}-X_{s-}$. For $t\geq 0$ and a Borel subset $A\subseteq \R^{n}$, we define the jump measure of $(X_{t})_{t\geq0}$ by
\begin{eqnarray*}
	N(t,A)
	&=& \text{the number of jumps during time $[0,t]$ of size in $A$} \\
	&=& |\{s\in[0,t]:\Delta X_{s}\in A\}|.
\end{eqnarray*}
Note that $N(t,A)$ is a Poisson random measure with intensity $dt\otimes \nu$. One can decompose $X_{t}$ into its continuous part $X_{t}^{c}$ and jump part
\begin{eqnarray*}
	X_{t}
	&=& X_{t}^{c}+\sum_{s:0\leq s\leq t}\Delta X_{s} \\
	&=&  bt+G(t)+\int_{|x|\geq 1}x N(t,dx)+\int_{|x|<1}x \wt{N}(t,dx), 
\end{eqnarray*}
where $b\in\R^{d}$, $G(t)$ is a Gaussian process and $\wt{N}(t,A):=N(t,A)-t\nu(A)$. Following the standard terminology, we call $\wt{N}(t,A)$ the compensated jump measure.

Next, recall It\^{o}'s formula for a general stochastic process with jumps. Let $(X_{t})_{t\geq0}$ be a L\'evy process with its jump measure $N(t,\cdot)$. Let $(Z_{t})_{t\geq0}$ be a $d$-dimensional stochastic process defined by
\begin{eqnarray}\label{semimartingale}
	Z_{t}=Z_{0}+M_{t}+A_{t}+\int_{0}^{t}\int_{\R^{d}}G(s,x)N(ds, dx)+\int_{0}^{t}\int_{\R^{d}}H(s,x)\widetilde{N}(ds, dx), 
\end{eqnarray}
where $M_{t}$ is a continuous square integrable local martingale, $A_{t}$ is a continuous adapted process of bounded variation with $A_{0}=0$, and $G(t,x)=(G_{1}(t,x),\cdots,G_{d}(t,x))$, $H(t,x)=(H_{1}(t,x),\cdots,H_{d}(t,x))$ are $d$-dimensional predictable processes.

\begin{theorem}[{It\^o's formula \cite[p. 66]{IW}}]\label{itoformula}
Let $(Z_{t})_{t\geq0}$ be given by \eqref{semimartingale} and $\varphi$ a $C^{2}$ function on $\R^{d}$. Assume that $G_{i}(t,x) H_{j}(t,x)=0$ for all $1\leq i,j\leq d$ and 
\begin{eqnarray*}
	\sup_{0\leq t\leq T}\sup_{x\in\R^{d}}|H(t,x)|<\infty
\end{eqnarray*} 
for all $T>0$,  almost surely. Then we have  
	\begin{eqnarray}\label{ito}
	\varphi(Z_{t})-\varphi(Z_{0})&=&\int_{0}^{t}\nabla \varphi(Z_{s})\cdot dM_{s}
		+\int_{0}^{t}\nabla \varphi(Z_{s})\cdot dA_{s}
	+\frac{1}{2}\int_{0}^{t}D^{2}\varphi(Z_{s})\cdot d[ M]_{s}\\
	&&+\int_{0}^{t}\int_{\R^{d}}\Paren{\varphi(Z_{s-}+G(s,x))-\varphi(Z_{s-})}N(ds, dx)\nonumber\\
	&&+\int_{0}^{t}\int_{\R^{d}}\Paren{\varphi(Z_{s-}+H(s,x))-\varphi(Z_{s-})}\widetilde{N}(ds, dx)\nonumber\\
	&&+\int_{0}^{t}\int_{\R^{d}}\Paren{\varphi(Z_{s-}+H(s,x))-\varphi(Z_{s-})-H(s,x)\cdot \nabla \varphi(Z_{s-})}\nu(dx)ds\nonumber
	\end{eqnarray}
	where $([M]_{t})_{t\geq0}$ is the quadratic variation of the martingale $(M_{t})_{t\geq0}$.
\end{theorem}

Throughout this paper, we assume that $(X_{t})_{t\geq0}$ is a pure jump L\'{e}vy process with  L\'{e}vy measure $\nu$ and that it  
satisfies the \emph{Hartman-Wintner condition}
\begin{equation}\label{HWcondition}\tag{HW}
	\lim_{|\xi|\to\infty}\frac{\Re(\psi(\xi))}{\log (1+ |\xi|)}=\infty.
\end{equation}
In \cite[Theorem 1]{KS}, V. Knopova and R. L. Schilling proved  that a L\'evy process $(X_{t})_{t\geq0}$ satisfies \eqref{HWcondition} if and only if  for all $t>0$, the transition density $p_{t}(x)$ exists, $p_{t}\in C^{\infty}(\R^{d})\cap L^{1}(\R^{d})$, and $\nabla p_{t}\in L^{1}(\R^{d})$. 
We define a semigroup $P_{t}f(x)=\E^{x}[f(X_{t})]$. It follows from the existence  of $p_{t}(x)$ that the semigroup $P_{t}$ is a $L^{p}$-contraction for $1\leq p\leq\infty$, meaning that $\norm{P_{t}f}_{p}\leq \norm{f}_{p}$,  for any $f\in L^{p}$. 
We also note that $P_{t}f\in L^{p}(\R^{d})\cap C^{\infty}(\R^{d})$ for $f\in L^{p}(\R^{d})$ and $1\leq p<\infty$. 
We define the infinitesimal generator $\cL$ for the semigroup $(P_{t})_{t\geq 0}$ by
\begin{eqnarray*}
	\cL f(x)=\lim_{t\downarrow 0}\frac{P_{t}f(x)-f(x)}{t}
\end{eqnarray*}
whenever the limit exists. It is well-known (see \cite[Theorem 31.5]{Sato}) that 
\begin{eqnarray}\label{infinitesimalgenerator}
	\cL f(x)=\int_{\R^{d}}(f(x+y)-f(x)-y\cdot \nabla f(x)\cdot \indi{B_{1}}(y))\nu(dy).
\end{eqnarray}

Finally, we recall the definition of a Fourier multiplier. Let $m:\R^{n}\to\C$ be a function in $L^{\infty}$ and $1\leq p\leq \infty$. For $f\in L^{2}\cap L^{p}$, we define an operator $T_{m}$ by $\wh{T_{m}f}(\xi)=m(\xi)\wh{f}(\xi)$. If $\norm{T_{m}f}_{p}\lesssim \norm{f}_{p}$ for all $f\in L^{2}\cap L^{p}$, then $T_{m}$ can be extended to all of  $L^{p}$ uniquely. We say $T_{m}$ is an {\it $L^p$-Fourier multiplier operator with symbol $m$}. For many of the classical examples of  $L^{p}$-Fourier multipliers, we refer the reader to \cite{Stein}.

\section{The Hardy-Stein Identity}

The purpose of this section is to give a proof of the Hardy-Stein identity based on It\^o's formula.
For $a,b\in\R$, $\varep>0$, and $p\in[1,\infty)$, we define
\begin{eqnarray}\label{eq:F}
	F(a,b;p)=|b|^{p}-|a|^{p}-pa|a|^{p-2}(b-a)
\end{eqnarray}  
and
\begin{eqnarray}\label{F-epsilon}
	F_{\varep}(a,b;p)
	=(b^{2}+\varep^{2})^{\frac{p}{2}}-(a^{2}+\varep^{2})^{\frac{p}{2}}
	-pa(a^{2}+\varep^{2})^{\frac{p-2}{2}}(b-a).
\end{eqnarray}
We note that $F(a,b;p)$ and $F_{\varep}(a,b;p)$ are the second-order Taylor reminders of the maps $x\mapsto |x|^{p}$  and $x\mapsto (x^{2}+\varep^{2})^{\frac{p}{2}}$ respectively. 
Since the maps are convex, it follows from Taylor's theorem that $F(a,b;p)\geq0$ and $F_{\varep}(a,b;p)\geq0$ for any $a,b\in\R$. 
\begin{theorem}[The Hardy-Stein identity]\label{taylorexpansion}
	Let $1<p<\infty$ and $F(a,b;p)$ be defined as in \eqref{eq:F}. If $f\in L^{p}(\R^{d})$, then we have
	\begin{eqnarray}\label{maineq}
		\int_{\R^{d}}|f(x)|^{p}dx
		=\int_{\R^{d}}\int_{0}^{\infty}\int_{\R^{d}}F(P_{t}f(x),P_{t}f(x+y);p)\nu (dy)dtdx.
	\end{eqnarray}
\end{theorem}

Again we note that our proof of this result does not require that $\nu$ is symmetric as is the case in \cite{BBL}.  Before we present the  proof of Theorem \ref{taylorexpansion}, we give the following lemmas.  The first lemma concerns basic properties of $F$ and $F_{\varep}$ which allow us to use a limiting argument when we consider the case $1<p<2$.  This lemma is proved in  \cite{BogdanBL}.

\begin{lemma}[{\cite[Lemma 6, p.198]{BogdanBL}}]\label{FKfunction}
	Let $p>1$, $F(a,b;p)=|b|^{p}-|a|^{p}-pa|a|^{p-2}(b-a)$, and $K(a,b;p)=(b-a)^{2}(|a|\vee|b|)^{p-2}$.
	Then we have 
	\begin{eqnarray*}
		c_{p} K(a,b;p)\leq F(a,b;p)\leq C_{p} K(a,b;p), 
	\end{eqnarray*}
	for some positive constants $c_{p}, C_{p}$ that depend only on $p$.
	If $1<p<2$, then we have
	\begin{eqnarray*}
		0\leq F_{\varep}(a,b;p)\leq \frac{1}{p-1}F(a,b;p) 
	\end{eqnarray*}
	for all $\varep>0$ and $a,b\in\R$.
\end{lemma}

Next lemma is an application of It\^o's formula. 

\begin{lemma}\label{martingaleformula}
Let $T>0$ and $1\leq p<\infty$.  For $f\in L^{p}(\R^{d})$ and $t\in[0,T]$, define $P_{t}f(x)=\E^{x}[f(X_{t})]$ and $Y_{t}=P_{T-t}f(X_{t})$. Then we have
\begin{eqnarray*}
	Y_{t}=Y_{0}+\int_{0}^{t}\int_{\R^{d}}(P_{T-s}f(X_{s-}+y)-P_{T-s}f(X_{s-}))\wt{N}(ds,dy)
\end{eqnarray*}
where $t\in[0,T]$ and $\wt{N}$ is the compensated jump measure of $(X_{t})_{t\geq0}$.
\end{lemma}
\begin{proof}
	Let $\varphi(s,x)=P_{T-s}f(x)$.  This is a $C^2$ function in $x$ by the assumption \eqref{HWcondition}.  Set $G(s,x)=x \indi{B_{1}^{c}}(x)$ and $H(s,x)=x \indi{B_{1}}(x)$, and apply It\^{o}'s formula \eqref{ito}  to get
	\begin{eqnarray*}
		\varphi(t,X_{t})-\varphi(0,X_{0})
		&=&\int_{0}^{t}\parder{}{s}\varphi(s,x)|_{x=X_{s-}}ds\\
		&&+\int_{0}^{t}\int_{|y|\geq 1}\Paren{\varphi(s, X_{s-}+y)-\varphi(s,X_{s-})}N(ds, dy)\\
		&&+\int_{0}^{t}\int_{|y|<1}\Paren{\varphi(s,X_{s-}+y)-\varphi(s,X_{s-})}\widetilde{N}(ds, dy)\\
		&&+\int_{0}^{t}\int_{|y|<1}\Paren{\varphi(s,X_{s-}+y)-\varphi(s,X_{s-})-y\cdot \nabla \varphi(s, X_{s-})}\nu(dy)ds.
	\end{eqnarray*}
	Since we have $\parder{}{s}\varphi(s,x)=-\cL P_{T-s}f(x)$, it follows from  \eqref{infinitesimalgenerator} that 
	\begin{eqnarray*}
		\int_{0}^{t}\parder{}{s}\varphi(s,x)|_{x=X_{s-}}ds
		&=&-\int_{0}^{t}\cL P_{T-s}f(X_{s-})ds\\
		&=&-\int_{0}^{t}\int_{\R^{d}}\Paren{
			\varphi(s,X_{s-}+y)-\varphi(s,X_{s-})-y\cdot \nabla \varphi(s, X_{s-})\cdot \indi{B_{1}}(y)
			}\nu(dy)ds\\
		&=&-\int_{0}^{t}\int_{|y|<1}\Paren{\varphi(s,X_{s-}+y)-\varphi(s,X_{s-})-y\cdot \nabla \varphi(s, X_{s-})}\nu(dy)ds\\
			&&-\int_{0}^{t}\int_{|y|\geq 1}\Paren{\varphi(s, X_{s-}+y)-\varphi(s,X_{s-})}\nu(dy)ds.
	\end{eqnarray*}
	From $\wt{N}(ds,dy)=N(ds,dy)-\nu(dy)ds$, we have
	\begin{eqnarray*}
		\varphi(t,X_{t})-\varphi(0,X_{0})
		=\int_{0}^{t}\int_{\R^{d}}\Paren{\varphi(s, X_{s-}+y)-\varphi(s,X_{s-})}\wt{N}(ds, dy) 
	\end{eqnarray*}
	as desired.
\end{proof}

\begin{lemma}\label{ultracontractivity}
	The semigroup $P_{t}$ defined by $P_{t}f(x)=\E^{x}[f(X_{t})]$ is ultracontractive on $L^p$, $1\leq p<\infty$.   That is, for every $t>0$, there exists a constant $C_{t}>0$ such that for all $f\in L^{p}(\R^{d})$,
	\begin{eqnarray}\label{eq:ultracontractive}
		\norm{P_{t}f}_{\infty}\leq C_{t}^{\frac{1}{p}}\norm{f}_{p}. 
	\end{eqnarray}
	Furthermore, $C_{t}$ converges to zero as $t$ tends to $\infty$.
\end{lemma}
\begin{proof}  Although not explicitly written, this  result follows from \cite{KS}. Since its proof is quite simple, we present it here for completeness. To prove the inequality, fix  $t>0$. 
	Note that $e^{-t\psi(\xi)}=\E^{0}[e^{i\xi\cdot X_{t}}]=(2\pi)^{d}\sF^{-1}(p_{t}(\cdot))(\xi)$. Since $p_{t}$ is in $L^{1}(\R^{d})$, one sees that $e^{-t\psi(\xi)}$ belongs to $L^{\infty}(\R^{d})$. We claim that $e^{-t\psi(\xi)}$ is in $L^{1}(\R^{d})$.  To see this, it suffices to show that $e^{-t\Re\psi(\xi)}\in L^{1}(\R^{d})$. Let $h:\R^{d}\to\R$ be a function satisfying $\Re \psi(\xi)=\log (1+|\xi|)h(\xi)$. Since we have  $h(\xi)\to\infty$ as $|\xi|\to\infty$ by the Hartman-Wintner condition \eqref{HWcondition}, there exists $R>0$ such that $th(\xi)>d+1$ holds whenever $|\xi|\geq R$. Let $B_{R}$ be the ball centered at 0 and  radius $R$. Denote its Lebesgue measure by $|B_{R}|$. Using the definition of $h$, one sees that
	\begin{eqnarray*}
		\int_{\R^{d}\setminus B_{R}}e^{-t\Re\psi(\xi)}d\xi
		=\int_{|\xi|\geq R}\frac{1}{(1+|\xi|)^{th(\xi)}}d\xi
		\leq \int_{|\xi|\geq R}\frac{1}{(1+|\xi|)^{d+1}}d\xi
	\end{eqnarray*}	
	Since we have
	\begin{eqnarray}\label{eq:exprepsi}
		e^{-t\Re\psi(\xi)}
		=|e^{-t\psi(\xi)}|
		=\Abs{\int_{\R^{d}}e^{i\xi\cdot x}p_{t}(x)dx}
		\leq 1,
	\end{eqnarray}
	we obtain
	\begin{eqnarray*}
		\int_{\R^{d}}e^{-t\Re\psi(\xi)}d\xi
		\leq  \int_{|\xi|\geq R}\frac{1}{(1+|\xi|)^{d+1}}d\xi
			+|B_{R}|
		<\infty,
	\end{eqnarray*}
	which implies that $e^{-t\Re\psi(\xi)}\in L^{1}(\R^{d})$.
	Since $p_{t}\in L^{1}(\R^{d})\cap L^{\infty}(\R^{d})$, we apply the Fourier inversion formula to obtain
	\begin{eqnarray*}
		p_{t}(x)
		=\frac{1}{(2\pi)^{d}}\sF(e^{-t\psi(\xi)})
		=\frac{1}{(2\pi)^{d}}\int_{\R^{d}}e^{-t\psi(\xi)}e^{-ix\cdot \xi}d\xi.
	\end{eqnarray*}
	Define
	\begin{eqnarray*}
		C_{t}=\frac{1}{(2\pi)^{d}}\int_{\R^{d}}e^{-t\Re\psi(\xi)}d\xi, 
	\end{eqnarray*}
	then it is obvious to see that $C_{t}$ is finite and $|p_{t}(x)|\leq C_{t}$ for all $x\in\R^{d}$. Using Jensen's inequality, we obtain that 
	\begin{eqnarray*}
		|P_{t}f(x)|
		=\Abs{\int_{\R^{d}}f(y)p_{t}(x,y)dy}
		\leq \Abs{\int_{\R^{d}}|f(y)|^{p}p_{t}(x,y)dy}^{\frac{1}{p}}
		\leq C_{t}^{\frac{1}{p}}\norm{f}_{p}, 
	\end{eqnarray*}
	for any $x\in\R^{d}$, which yields \eqref{eq:ultracontractive}. 
	
	We now prove the second assertion that $C_{t}\to 0$ as $t\to\infty$. First, we note that $\Re\psi(\xi)$ is nonnegative by \eqref{eq:exprepsi} and in fact the Lebesgue measure of the set $\{\xi: \Re\psi(\xi)=0\}$ is zero (see \S3 in \cite{BanBog}).  Thus $e^{-t\Re\psi(\xi)}$ tends to 0, a.e., as $t\to\infty$. Since $e^{-t\Re\psi(\xi)}$ is integrable for all $t\geq 1$ and bounded by $e^{-\Re\psi(\xi)}$, it follows from the dominated convergence theorem that 
	\begin{eqnarray*}
		\lim_{t\to\infty}C_{t}
		=\lim_{t\to\infty}\frac{1}{(2\pi)^{d}}\int_{\R^{d}}e^{-t\Re\psi(\xi)}d\xi
		=0.
	\end{eqnarray*}
\end{proof}

We are ready to prove the Hardy-Stein identity.
\begin{proof}[Proof of Theorem \ref{taylorexpansion}]
We begin with the case $p\geq 2$. Fix $T>0$ and define $\varphi(x)=|x|^{p}$, $Y_{t}:=P_{T-t}f(X_{t})$, and $H(s,y)=P_{T-s}f(X_{s-}+y)-P_{T-s}f(X_{s-})$. Let $0<T_{0}<T$, then it follows from Lemma \ref{martingaleformula} that $Y_{t}$ can be written in the form
\begin{eqnarray*}
	Y_{t}=Y_{0}+\int_{0}^{t}\int_{\R^{d}}H(s,y)\wt{N}(ds,dy)
\end{eqnarray*}
for any $0\leq t\leq T_{0}$.
Note that $H(s,y)$ is uniformly bounded in $s\in[0,T_{0}]$ and $y\in\R^{d}$, by Lemma \ref{ultracontractivity}. Applying It\^{o}'s formula  to $\varphi(Y_{t})$, we obtain
\begin{eqnarray}\label{eq:proof3_1_1}
	\varphi(Y_{t})-\varphi(Y_{0})
	&=&\int_{0}^{t}\int_{\R^{d}}\Paren{\varphi(Y_{s-}+H(s,y))-\varphi(Y_{s-})}\widetilde{N}(ds, dy)\\
	&&+\int_{0}^{t}\int_{\R^{d}}\Paren{\varphi(Y_{s-}+H(s,y))-\varphi(Y_{s-})-H(s,y)\cdot \nabla \varphi(Y_{s-})}\nu(dy)ds\nonumber
\end{eqnarray}
for all $0\leq t\leq T_{0}$.
Note that $Y_{s-}+H(s,y)=P_{T-s}f(X_{s-}+y)$ and $Y_{s-}=P_{T-s}f(X_{s-})$. If we take expectation of both sides in \eqref{eq:proof3_1_1}, the first integral on the right hand side vanishes because it is a martingale. We note  that 
\begin{eqnarray*}
	&&\varphi(Y_{s-}+H(s,y))-\varphi(Y_{s-})-H(s,y)\cdot \nabla \varphi(Y_{s-})\\
	&=&|P_{T-s}f(X_{s-}+y)|^{p}-|P_{T-s}f(X_{s-})|^{p}\\
	&&-pP_{T-s}f(X_{s-})|P_{T-s}f(X_{s-})|^{p-2}(P_{T-s}f(X_{s-}+y)-P_{T-s}f(X_{s-}))\\
	&=& F(P_{T-s}f(X_{s-}+y), P_{T-s}f(X_{s-});p).
\end{eqnarray*} 
Putting $t=T_{0}$ and taking the expectation of both sides, we have
\begin{eqnarray}\label{eq:proof3_1_2}
	\E^{x}|Y_{T_{0}}|^{p}-\E^{x}|Y_{0}|^{p}
	=\Ebrx{x}{\int_{0}^{T_{0}}\int_{\R^{d}}F(P_{T-s}f(X_{s-}+y), P_{T-s}f(X_{s-});p)\nu(dy)ds}.
\end{eqnarray}
Since we have $\int_{\R^{d}}\E^{x}|Y_{T_{0}}|^{p}dx=\norm{P_{T-T_{0}}f}_{p}^{p}$ and $\int_{\R^{d}}\E^{x}|Y_{0}|^{p}dx=\norm{P_{T}f}_{p}^{p}$, we integrate both sides in \eqref{eq:proof3_1_2} to see
\begin{eqnarray*}
	\norm{P_{T-T_{0}}f}_{p}^{p}-\norm{P_{T}f}_{p}^{p}
	&=& \int_{\R^{d}}\Ebrx{x}{\int_{0}^{T_{0}}\int_{\R^{d}}F(P_{T-s}f(X_{s-}+y), P_{T-s}f(X_{s-});p)\nu(dy)ds}dx \\
	&=& \int_{\R^{d}}\int_{\R^{d}}\int_{0}^{T_{0}}\int_{\R^{d}}F(P_{T-s}f(z+y), P_{T-s}f(z);p)p_{s}(x,z)\nu(dy)dsdzdx.\\
	&=& \int_{\R^{d}}\int_{0}^{T_{0}}\int_{\R^{d}}F(P_{T-s}f(z+y), P_{T-s}f(z);p)\nu(dy)dsdz.
\end{eqnarray*}  
Since $F(a,b;p)$ is nonnegative and $\norm{P_{T-T_{0}}f}_{p}^{p}\to\norm{f}_{p}^{p}$ as $T_{0}\to T$, the monotone convergence theorem yields
\begin{eqnarray}\label{Ttoinfinity}
	\norm{f}_{p}^{p}-\norm{P_{T}f}_{p}^{p}
	&=& \int_{\R^{d}}\int_{0}^{T}\int_{\R^{d}}F(P_{s}f(z+y), P_{s}f(z);p)\nu(dy)dsdz.
\end{eqnarray}
By Lemma \ref{ultracontractivity}, one sees that $P_{T}f(x)\to 0$, as $T\to\infty$ for each $x\in\R^{d}$. This and the fact that  $\norm{P_{T}f}_{p}\leq \norm{f}_{p}$ for all $T>0$ gives that $\lim_{T\to\infty}\norm{P_{T}f}_{p}^{p}= 0$; since $F(a,b;p)$ is nonnegative, the right hand side of \eqref{Ttoinfinity} converges to the desired limit, which proves the result for $p\geq 2$.

We now consider the case $1<p<2$. Let $\varep>0$.  If we follow the same argument as in the case $p>2$ with the function $\varphi(x)=(|x|^{2}+\varep^{2})^{\frac{p}{2}}$, we arrive at
\begin{eqnarray*}\label{casepgreater2}
	\int_{\R^{d}}\Paren{(|f(x)|^{2}+\varep^{2})^{\frac{p}{2}}-(|P_{T}f(x)|^{2}+\varep^{2})^{\frac{p}{2}}}dx
	= \int_{\R^{d}}\int_{0}^{T}\int_{\R^{d}}F_{\varep}(P_{s}f(z+y), P_{s}f(z);p)\nu(dy)dsdz, 
\end{eqnarray*}
where $F_{\varep}$ is the function in \eqref{F-epsilon}. 

First, look at the left hand side of this identity.  Since the function $x\mapsto x^{\frac{p}{2}}$ is $\frac{p}{2}$-H\"{o}lder continuous on $[0,\infty)$ for $1<p<2$, we have $(|f(x)|^{2}+\varep^{2})^{\frac{p}{2}}-\varep^{p}\leq C_{p}|f(x)|^{p}$ and $(|P_{T}f(x)|^{2}+\varep^{2})^{\frac{p}{2}}-\varep^{p}\leq C_{p}|P_{T}f(x)|^{p}$.
Thus, the left hand side converges to $\norm{f}_{p}^{p}-\norm{P_{T}f}_{p}^{p}$ as $\varep\to 0$ by the dominated convergence theorem.  On the other hand,  $0\leq F_{\varep}(a,b;p)\to F(a,b;p)$, as $\varep\to 0$,  and $0\leq F_{\varep}(a,b;p)\leq \frac{1}{p-1}F(a,b;p)$, by Lemma \ref{FKfunction}. Since the integral
\begin{eqnarray*}
	I(\varep,T):=\int_{\R^{d}}\int_{0}^{T}\int_{\R^{d}}F_{\varep}(P_{s}f(z+y), P_{s}f(z);p)\nu(dy)dsdz
\end{eqnarray*}
is bounded for each $\varep>0$, Fatou's lemma and the dominated convergence theorem give (see \cite[p.199]{BogdanBL}) that
\begin{eqnarray*}
	I(\varep,T)\to\int_{\R^{d}}\int_{0}^{T}\int_{\R^{d}}F(P_{s}f(z+y), P_{s}f(z);p)\nu(dy)dsdz,
\end{eqnarray*}
as $\varep\to 0$. We now finish the proof by letting $T\to \infty$.
\end{proof}

One can easily see that the proof above gives the following more general result for martingales of which Theorem 
\ref{taylorexpansion} is a special case. 

\begin{theorem}[A Hardy-Stein identity for martingales]\label{martHardyStein} 
Let $H(t,x)$ be a predictable process. Suppose that for each $T>0$, $H$ is uniformly bounded in $t\in[0,T]$ and $x\in\R^{d}$. If a martingale $M_{t}$ is given by 
\begin{eqnarray*}
	M_{t}=M_{0}+\int_{0}^{t}\int_{\R^{d}}H(s,y)\wt{N}(ds,dy),
\end{eqnarray*}
then for  $1<p<\infty$, we have 
\begin{eqnarray}\label{marhardystein}
	\E|M_{\infty}|^{p}-\E|M_{0}|^{p}
	=\int_{0}^{\infty}\int_{\R^{d}}\E[F(M_{s-},M_{s-}+H(s,y);p)]\nu(dy)ds.
\end{eqnarray} 
\end{theorem}

\section{Fourier multipliers and square functions}

The main application of the results in \cite{BBL} were to show the boundedness of the Fourier multipliers introduced in  \cite{BBB} on $L^p$, $1<p<\infty$, without appealing to martingale transforms. Of course, a disadvantage of such a proof is that we do not obtain the sharp bounds given in \cite{BanBog} and  \cite{BBB}, which follow from Burkholder's sharp inequalities.  In addition, the Littlewood-Paley inequalities proved in \cite{BBL} only apply to symmetric L\'evy processes and therefore the Fourier multiplier proof given there also has this restriction. In this section, we provex, via symmetrization of the Littlewood-Paley inequalities, the general result for Fourier multipliers. 

Let $(X_{t})_{t\geq0}$ be a pure jump L\'evy process with L\'evy measure $\nu$ and $P_{t}$ be the semigroup defined by $P_{t}f(x)=\E^{x}[f(X_{t})]$.
Let $\phi:(0,\infty)\times \R^{d}\to\R$ be a bounded function and $1<p,q<\infty$ with $\frac{1}{p}+\frac{1}{q}=1$. We define $\Lambda_{\phi}(f,g)$ for $f\in L^{2}(\R^{d})\cap L^{p}(\R^{d})$ and $g\in L^{2}(\R^{d})\cap L^{q}(\R^{d})$ by
\begin{eqnarray}\label{eq:lambda}
	\Lambda_{\phi}(f,g)
	=\int_{\R^{d}}\int_{0}^{\infty}\int_{\R^{d}}
	(P_{t}f(x+y)-P_{t}f(x))(P_{t}g(x+y)-P_{t}g(x))\phi(t,y)
	\nu(dy)dtdx.
\end{eqnarray}
Let $m:\R^{d}\to\C$ be a function. The Fourier multiplier operator with symbol $m$ is denoted by $\cM_{m}$. Note that $\cM_{m}$ is determined by $\sF(\cM_{m} f)(\xi)=m(\xi)\wh{f}(\xi)$. For $f,g\in L^{2}(\R^{d})$, we denote by $\brk{f,g}=\int_{\R^{d}}f(x)g(x)dx$. By Parseval's formula \eqref{parseval}, we have
\begin{eqnarray*}
	\brk{\cM_{m}f,g}
	&=&\int_{\R^{d}}\cM_{m}f(x)g(x)dx
	=\frac{1}{(2\pi)^{d}}\int_{\R^{d}}\sF(\cM_{m}f)(\xi)\ol{\sF(g)(\xi)}d\xi\\
	&=&\frac{1}{(2\pi)^{d}}\int_{\R^{d}}m(\xi)\wh{f}(\xi)\ol{\wh{g}(\xi)}d\xi
\end{eqnarray*}
We are ready to state our result on Fourier multipliers. 
\begin{theorem}\label{Fouriermultiplier}
	Let $\phi:(0,\infty)\times \R^{d}\to\R$ be a bounded function, $p\in (1,\infty)$ and $q$ the conjugate exponent of $p$. Let $\Lambda_{\phi}(f,g)$ be defiend as in \eqref{eq:lambda}.
	Then, $\Lambda_{\phi}(f,g)$ is absolutely convergent for $f\in L^{2}(\R^{d})\cap L^{p}(\R^{d})$ and $g\in L^{2}(\R^{d})\cap L^{q}(\R^{d})$. Furthermore, there is a unique bounded linear operator $S_{\phi}$ on $L^{p}(\R^{d})$ such that $\Lambda_{\phi}(f,g)=\brk{S_{\phi}(f),g}$ and $S_{\phi}=\cM_{m_{\phi}}$ with symbol $m_{\phi}$ given by
	\begin{eqnarray*}
		m_{\phi}(\xi)
		=\int_{0}^{\infty}\int_{\R^{d}}
		|e^{i\xi\cdot y}-1|^{2}e^{-2t\Re(\psi(\xi))}\phi(t,y)\nu(dy)dt.
	\end{eqnarray*}
\end{theorem}

When $\nu$ is symmetric, this result was proved in \cite{BBL} as an application of the boundedness on $L^p$ of the Littlewood-Paley square functions which itself was the main application of the Hardy-Stein inequality, completely bypassing the martingale transform arguments used earlier. The question left open in \cite{BBL} was whether  Littlewood-Paley arguments can be used to prove the result for general $\nu$.  We answer this in the affirmative. Our proof relies on the symmetrization technique. 

Let us introduce the dual process and the symmetrization of the L\'evy process $(X_{t})_{t\geq 0}$ with the L\'evy measure $\nu$. Let $(\wh{X}_{t})_{t\geq0}$ be a c\`{a}dl\`{a}g stochastic process having the same finite dimensional distribution as $(-X_{t})_{t\geq0}$, and independent of $(X_{t})_{t\geq0}$. The process $(\wh{X}_{t})_{t\geq0}$ is said to be the \emph{dual process} of $(X_{t})_{t\geq0}$. Note that $(\wh{X}_{t})_{t\geq0}$ is a L\'{e}vy process with triplet $(0,0,\nu(-dx))$. We define its semigroup by $\wh{P}_{t}f(x):=\E^{x}[f(\wh{X}_{t})]$. Note that for any Borel function $f$ and $g$, we have
\begin{eqnarray*}
	\int_{\R^{d}}P_{t}f(x)g(x)dx=\int_{\R^{d}}f(x)\wh{P}_{t}g(x)dx,
\end{eqnarray*}
which explains why $(\wh{X}_{t})_{t\geq0}$ is called the dual of $(X_{t})_{t\geq 0}$.

Let $\wt{X}_{t}:=X_{\frac{t}{2}}+\wh{X}_{\frac{t}{2}}$ for $t\geq 0$. We define $\wt{\psi}(\xi):=\Re(\psi(\xi))$ and $\wt{\nu}(B):=\frac{1}{2}(\nu(B)+\nu(-B))$ for any measurable set $B$ in $\R^{d}$. Since we have
\begin{eqnarray*}
	\E[e^{i\xi\cdot \wt{X}_{t}}]
	&=& \E[e^{i\xi\cdot X_{\frac{t}{2}}}]\E[e^{i\xi\cdot \wh{X}_{\frac{t}{2}}}]\\
	&=& e^{-\frac{t}{2}\psi(\xi)}e^{-\frac{t}{2}\psi(-\xi)}\\
	&=& e^{-t\wt{\psi}(\xi)}
\end{eqnarray*}
and
\begin{eqnarray*}
	\wt{\psi}(\xi)
	&=&\int_{\R^{d}}(1-\cos(\xi\cdot y))\nu(dy)\\
	&=&\int_{\R^{d}}(1-\cos(\xi\cdot y))\wt{\nu}(dy)\\
	&=&\int_{\R^{d}}(1-e^{i\xi\cdot y}+i\xi\cdot y1_{|y|\leq 1})\wt{\nu}(dy),
\end{eqnarray*}
the process $\wt{X}_{t}$ is a L\'{e}vy process with characteristic exponent $\wt{\psi}(\xi)$ and the L\'evy measure $\wt{\nu}$. We say $\wt{X}_{t}$ the symmetrization of $X_{t}$. Define $\wt{P}_{t}f(x)=\E^{x}[f(\wt{X}_{t})]$. The Fourier transform of $\wt{P}_{t}f$ is given by 
\begin{eqnarray*}
	\sF(\wt{P}_{t}f)(\xi)
	=e^{-t\wt{\psi}(\xi)}\wh{f}(\xi)
	=e^{-t\Re(\psi(\xi))}\wh{f}(\xi).
\end{eqnarray*}

Since $\wt{X}_{t}$ is a pure jump symmetric L\'{e}vy process and the measure $\wt{\nu}$ satisfies \eqref{HWcondition} condition, it leads us to apply the result of \cite{BBL} for the symmetrization $\wt{X}_{t}$. In particular, we obtain two side estimates for the square functions of $\wt{X}_{t}$. We define the square functions of the symmetrized process $\wt{X}_{t}$ by
\begin{eqnarray}\label{defsquareftn}
	&&\wt{G}(f)(x)=\Paren{\int_{0}^{\infty}\int_{\R^{d}}\Abs{\wt{P_{t}}f(x+y)-\wt{P_{t}}f(x)}^{2}\wt{\nu}(dy)dt}^{\frac{1}{2}},\\
	&&\wt{G}_{\ast}(f)(x)=\Paren{\int_{0}^{\infty}\int_{A(t,x,f)}\Abs{\wt{P_{t}}f(x+y)-\wt{P_{t}}f(x)}^{2}\wt{\nu}(dy)dt}^{\frac{1}{2}}\label{symG*}
\end{eqnarray}
where $A(t,x,f):=\{y\in\R^{d}:|\wt{P_{t}}f(x)|>|\wt{P_{t}}f(x+y)|\}$. The following lemma is found in \cite[Theorem 4.1, Corollary 4.4 ]{BBL}.
\begin{lemma}\label{squarefunctions}
	Let $2\leq p<\infty$ and $f\in L^{p}(\R^{d})$. Then there are constants $c_{p}$ and $C_{p}$ depending only on $p$ such that 
	\begin{eqnarray*}
		c_{p}\norm{f}_{p}\leq \norm{\wt{G}(f)}_{p}\leq C_{p}\norm{f}_{p}. 
	\end{eqnarray*}
	If $1 < p<\infty$ and $f\in L^{p}(\R^{d})$, then we have 
	\begin{eqnarray*}
		d_{p}\norm{f}_{p}\leq \norm{\wt{G}_{\ast}(f)}_{p}\leq D_{p}\norm{f}_{p}. 
	\end{eqnarray*}
	for some $d_{p}$ and $D_{p}$ depending only on $p$.
\end{lemma}

For a function $f$ and a measure $\mu$, the essential supremum of $f$ with respect to the measure $\mu$ is denoted by $\norm{f}_{\infty,\nu}$. The following lemma plays an important role in the symmetrization argument. 
\begin{lemma}\label{msrlemma}
	Let $\ol{\nu}(B):=\frac{1}{2}(\nu(B)-\nu(-B))$ for any measurable sets $B\subseteq \R^{d}$. Then, there is a measurable function $r(y)$ such that $\ol{\nu}(dy)=r(y)\wt{\nu}(dy)$. Furthermore, the function $r(y)$ is bounded $\wt{\nu}$-a.s. with $\norm{r}_{\infty,\wt{\nu}}\leq 1$.
\end{lemma}
\begin{proof}
	The existence of such a measurable function follows from the Radon-Nikodym theorem. Let us explain this in detail. Note that $\nu$ is $\sigma$-finite since $\nu(\{0\})=0$ and $\int_{\R^{d}}(1\wedge|x|^{2})\nu(dx)<\infty$. So are $\wt{\nu}$ and $\ol{\nu}$. Suppose that $B\subseteq \R^{d}$ is a measurable set such that $\wt{\nu}(B)=0$. Since $\nu$ is a positive measure, we have $\nu(B)=\nu(-B)=0$, which implies $\ol{\nu}(B)=0$. Thus $\ol{\nu}$ is absolutely continuous with respect to $\wt{\nu}$. By the Radon-Nikodym theorem, we conclude that there is a measurable function $r(y)$ such that $\ol{\nu}(dy)=r(y)\wt{\nu}(y)$.
	
	To see $r(y)$ is bounded, we consider the set $B^{\varep}:=\{y\in\R^{d}:|r(y)|>1+\varep\}$ for an arbitrary $\varep>0$. From the relation $\ol{\nu}(dy)=r(y)\wt{\nu}(y)$ obtained above, we have $\ol{\nu}(B^{\varep})>(1+\varep)\wt{\nu}(B^{\varep})$. It then yields $\varep \nu(B^{\varep})+(2+\varep)\nu(-B^{\varep})<0$ so that $\nu(B^{\varep})=\nu(-B^{\varep})=0$. Therefore, $r(y)$ is bounded $\wt{\nu}$-a.s. and $\norm{r}_{\infty,\wt{\nu}}\leq 1$.
\end{proof}

\begin{proof}[Proof of Theorem \ref{Fouriermultiplier}]
	The first argument is directly obtained by Theorem \ref{taylorexpansion}. Indeed, since $F(a,b;2)=|a-b|^{2}$, Theorem \ref{taylorexpansion} yields that
	\begin{eqnarray*}
		\norm{f}_{2}^{2}
		&=&\int_{\R^{d}}\int_{0}^{\infty}\int_{\R^{d}}F(P_{t}f(x),P_{t}f(x+y);2)\nu (dy)dtdx\\
		&=&\int_{\R^{d}}\int_{0}^{\infty}\int_{\R^{d}}|P_{t}f(x)-P_{t}f(x+y)|^{2}\nu (dy)dtdx.
	\end{eqnarray*}
	It then follows from the Cauchy-Schwartz inequality that
	\begin{eqnarray*}
		|\Lambda_{\phi}(f,g)|
		&\leq & \norm{\phi}_{\infty}\int_{\R^{d}}\int_{0}^{\infty}\int_{\R^{d}}|P_{t}f(x+y)-P_{t}f(x)||P_{t}g(x+y)-P_{t}g(x)|\nu(dy)dtdx\\
		&\leq & \norm{\phi}_{\infty}\norm{f}_{2}\norm{g}_{2}.
	\end{eqnarray*}
	Since $f,g\in L^{2}(\R^{d})$, Theorem \ref{taylorexpansion} implies that $\Lambda_{\phi}(f,g)$ is absolutely convergent.
	To see the second assertion, we use Parseval's formula \eqref{parseval} so that
	\begin{eqnarray*}
		&&\Lambda_{\phi}(f,g)\\
		&=& \frac{1}{(2\pi)^{d}}\int_{0}^{\infty}\int_{\R^{d}}\int_{\R^{d}}
		\sF(P_{t}f(\cdot+y)-P_{t}f(\cdot))(\xi)
		\ol{\sF(P_{t}g(\cdot+y)-P_{t}g(\cdot))(\xi)}\phi(t,y)
		d\xi\nu(dy)dt\\
		&=& \frac{1}{(2\pi)^{d}}\int_{0}^{\infty}\int_{\R^{d}}\int_{\R^{d}}
		(e^{i \xi\cdot y}-1)e^{-t\psi(\xi)}\wh{f}(\xi)\ol{(e^{i \xi\cdot y}-1)e^{-t\psi(\xi)}\wh{g}(\xi)}\phi(t,y)
		d\xi\nu(dy)dt\\
		&=& \frac{1}{(2\pi)^{d}}\int_{0}^{\infty}\int_{\R^{d}}\int_{\R^{d}}
		|e^{ i \xi\cdot y}-1|^{2}e^{-2t\Re(\psi(\xi))}\wh{f}(\xi)\ol{\wh{g}(\xi)}\phi(t,y)
		d\xi\nu(dy)dt
	\end{eqnarray*}
	where $\psi(\xi)$ is the characteristic exponent of $(X_{t})_{t\geq0}$. In the second equality, we have used the fact that
	\begin{eqnarray*}
		\sF(P_{t}f(\cdot+y)-P_{t}f(\cdot))(\xi)
		=(e^{ i \xi\cdot y}-1)\sF(P_{t}f)(\xi)
		=(e^{ i \xi\cdot y}-1)e^{-t\psi(\xi)}\wh{f}(\xi).
	\end{eqnarray*}
	Let $\wt{\nu}(B)=\frac{1}{2}(\nu(B)+\nu(-B))$ and $\ol{\nu}(B)=\frac{1}{2}(\nu(B)-\nu(-B))$ for any measurable set $B$ in $\R^{d}$. By Lemma \ref{msrlemma}, there is a measurable function $r(y)$ such that $\ol{\nu}(dy)=r(y)\wt{\nu}(dy)$ with $\norm{r}_{\infty,\wt{\nu}}\leq 1$. Using $\nu=\wt{\nu}+\ol{\nu}$, we have
	\begin{eqnarray*}
		\Lambda_{\phi}(f,g)
		&=&\frac{1}{(2\pi)^{d}}\int_{0}^{\infty}\int_{\R^{d}}\int_{\R^{d}}
		|e^{i \xi\cdot y}-1|^{2}e^{-2t\Re(\psi(\xi))}\wh{f}(\xi)\ol{\wh{g}(\xi)}\phi(t,y)
		d\xi\wt{\nu}(dy)dt\\
		&&+\frac{1}{(2\pi)^{d}}\int_{0}^{\infty}\int_{\R^{d}}\int_{\R^{d}}
		|e^{i \xi\cdot y}-1|^{2}e^{-2t\Re(\psi(\xi))}\wh{f}(\xi)\ol{\wh{g}(\xi)}\phi(t,y)
		d\xi\ol{\nu}(dy)dt.
	\end{eqnarray*}
	If we define $\eta(t,y)=\phi(t,y)(1+r(y))$, then $\eta$ is bounded $\wt{\nu}$-a.s; thus, we obtain
	\begin{eqnarray*}
		\Lambda_{\phi}(f,g)
		=\frac{1}{(2\pi)^{d}}\int_{0}^{\infty}\int_{\R^{d}}\int_{\R^{d}}
		|e^{ i \xi\cdot y}-1|^{2}e^{-2t\Re(\psi(\xi))}\wh{f}(\xi)\ol{\wh{g}(\xi)}\eta(t,y)
		d\xi\wt{\nu}(dy)dt.
	\end{eqnarray*}
	We consider $\wt{X}_{t}$ and $\wt{P}_{t}$, the symmetrization of $X_{t}$ and $P_{t}$. Since the characteristic exponent of $\wt{X}_{t}$ is the real part of $\psi(\xi)$, $\wt{\psi}(\xi)=\Re(\psi(\xi))$, and 
	\begin{eqnarray*}
		\sF(\wt{P}_{t}f(\cdot+y)-\wt{P}_{t}f(\cdot))(\xi)
		=(e^{ i \xi\cdot y}-1)\sF(\wt{P}_{t}f)(\xi)
		=(e^{ i \xi\cdot y}-1)e^{-t\Re(\psi(\xi))}\wh{f}(\xi),
	\end{eqnarray*}
	it follows from Parseval's formula \eqref{parseval} that
	\begin{eqnarray}\label{symmlambda}
		\Lambda_{\phi}(f,g)
		&=& \int_{\R^{d}}\int_{0}^{\infty}\int_{\R^{d}}
		(\wt{P_{t}}f(x+y)-\wt{P_{t}}f(x))(\wt{P_{t}}g(x+y)-\wt{P_{t}}g(x))\eta(t,y)\wt{\nu}(dy)dtdx\nonumber\\
		&=:& \wt{\Lambda}_{\eta}(f,g).
	\end{eqnarray}
	To show the boundedness of $\Lambda_{\phi}(f,g)$, we use square functions defined in \eqref{defsquareftn}. It is enough to show the case $p>2$ and $1<q<2$. Note that $\norm{\eta}_{\infty,\wt{\nu}}$ is finite and $\norm{\eta}_{\infty, \wt{\nu}}\leq 2\norm{\phi}_{\infty}$. Applying Cauchy-Schwartz and H\"older inequalities, we have
	\begin{eqnarray*}
		|\wt{\Lambda}_{\eta}(f,g)|
		&\leq & \norm{\eta}_{\infty, \wt{\nu}} \int_{\R^{d}}\int_{0}^{\infty}\int_{\R^{d}}
		|\wt{P_{t}}f(x+y)-\wt{P_{t}}f(x)||\wt{P_{t}}g(x+y)-\wt{P_{t}}g(x)|\wt{\nu}(dy)dtdx\\
		&\leq & 2\norm{\eta}_{\infty, \wt{\nu}} \int_{\R^{d}}\int_{0}^{\infty}\int_{A(t,x,g)}
		|\wt{P_{t}}f(x+y)-\wt{P_{t}}f(x)||\wt{P_{t}}g(x+y)-\wt{P_{t}}g(x)|\wt{\nu}(dy)dtdx\\
		&\leq & 2\norm{\eta}_{\infty, \wt{\nu}} \int_{\R^{d}}\wt{G}(f)(x)\wt{G}_{\ast}(g)(x)dx\\
		&\leq & 2\norm{\eta}_{\infty, \wt{\nu}} \norm{\wt{G}(f)}_{p}\norm{\wt{G}_{\ast}(g)}_{q}
	\end{eqnarray*}
	where $A(t,x,g):=\{y\in\R^{d}:|\wt{P_{t}}g(x)|>|\wt{P_{t}}g(x+y)|\}$. It follows from Lemma \ref{squarefunctions} and \eqref{symmlambda} that
	\begin{eqnarray*}
		\Lambda_{\phi}(f,g)\leq 4C_{p}D_{q}\norm{\phi}_{\infty}\norm{f}_{p}\norm{g}_{q}.
	\end{eqnarray*}
	Therefore, the Riesz representation theorem yields that there is a unique linear operator $S_{\phi}$ satisfying $\Lambda_{\phi}(f,g)=\brk{S_{\phi}(f),g}$.
\end{proof}

\end{document}